\documentclass[11pt]{amsart}

\usepackage{amsmath,amssymb,color,multicol,setspace}
\usepackage[utf8,utf8x]{inputenc}

\usepackage{longtable}
\usepackage{xy,amscd} 
\usepackage{mathdots}
\usepackage{epsfig}
\usepackage{rotating}
\usepackage{xspace}
\usepackage{float}
\usepackage{enumitem} 
\xyoption{all}
\usepackage{tikz}
\usetikzlibrary{arrows,decorations.pathmorphing,decorations.pathreplacing,positioning,shapes.geometric,shapes.misc,decorations.markings,decorations.fractals,calc,patterns}

\newcommand{\blue}{\color{blue}}
\newcommand{\red}{\color{red}}

\newtheorem{Lemma}{Lemma}[section]
\newtheorem{Theorem}[Lemma]{Theorem}
\newtheorem{Proposition}[Lemma]{Proposition}
\newtheorem{Corollary}[Lemma]{Corollary}

\theoremstyle{definition}
\newtheorem{Definition}[Lemma]{Definition}

\newtheorem{Example}[Lemma]{Example}

\numberwithin{equation}{section}

\title[Weak friezes and frieze pattern determinants]
{Weak friezes and frieze pattern determinants}

\author{Thorsten Holm}
\address{Thorsten Holm, Leibniz Universit\"at Hannover,
Institut f\"ur Algebra, Zahlentheorie und Diskrete Mathematik,
Fakult\"at f\"ur Mathematik und Physik,
Welfengarten 1,
D-30167 Hannover, Germany}
\email{holm@math.uni-hannover.de}
\urladdr{https://www.iazd.uni-hannover.de/de/holm/}

\author{Peter J{\o}rgensen}
\address{Peter J{\o}rgensen, 
Department of Mathematics,
Aarhus University,
Ny Munkegade 118,
8000 Aarhus C,
Denmark}
\email{peter.jorgensen@math.au.dk}
\urladdr{https://sites.google.com/view/peterjorgensen}



\begin{document}

\begin{abstract}
Frieze patterns have been introduced by Coxeter in the 1970's and have recently attracted
renewed interest due to their close connection with Fomin-Zelevinsky's cluster algebras. 
Frieze patterns can be interpreted as assignments of values to the diagonals of
a triangulated polygon satisfying certain conditions for crossing diagonals (Ptolemy relations).
Weak friezes, as introduced by \c{C}anak\c{c}{\i} and J{\o}rgensen, are generalizing this concept
by allowing to glue dissected polygons so that the Ptolemy relations only have to be satisfied
for crossings involving one of the gluing diagonals. 

To any frieze pattern one can associate a symmetric matrix using a triangular fundamental domain of the 
frieze pattern in the upper and lower half of the matrix and putting zeroes on the diagonal. 
Broline, Crowe and Isaacs have found a formula for the determinants of these matrices and their 
work has later been generalized in various directions by other authors. These frieze pattern determinants are the 
main focus of our paper. As our main result we show that this determinant behaves well with
respect to gluing weak friezes: the determinant is the product of the determinants for the 
pieces glued, up to a scalar factor coming from the gluing diagonal. Then we give several applications 
of this result, showing that formulas from the literature, obtained by Broline-Crowe-Isaacs, Baur-Marsh, 
Bessenrodt-Holm-J{\o}rgensen and Maldonado can all be obtained as consequences of our result. 
\end{abstract}

\maketitle

\section{Introduction}

Frieze patterns have been introduced by Coxeter \cite{Cox71} around 1970 as certain patterns
of numbers. Soon afterwards 
Conway and Coxeter gave a geometric-combinatorial model for frieze patterns over positive
integers by showing that these are in bijection with triangulations of polygons. Back then, one
might have considered frieze patterns to be merely a nice piece of recreational mathematics. This 
changed completely some twenty years ago after Fomin and Zelevinsky came up with the 
fundamental and far-reaching notion of cluster algebras \cite{FZ02}. The exchange 
relations for cluster algebras mimic the conditions defining frieze patterns, e.g. for Dynkin type
$A$ one gets the cluster variables as entries in a frieze pattern.
This has been the starting point for a revived interest in frieze patterns with many new
developments in recent years. Frieze patterns now form a nexus between different areas 
of mathematics, like algebra, combinatorics and geometry, see the survey \cite{MG15}.

The classic notion of frieze pattern has been varied and generalized in different ways, via
finite and infinite frieze patterns, 
$SL_k$-frieze patterns or $SL_k$-tilings of the plane, each topic
with an extensive recent literature. For our purposes, the concept of frieze patterns with coefficients
is relevant, see \cite{CHJ20} for details. 
A frieze pattern with coefficients over a field $K$ of height $n-3$ ($n\ge 3$ a 
natural number) 
is an infinite array of the form as in Figure \ref{fig:friezecoeff}
\begin{figure}
$$
\begin{array}{cccccccccc}
 & \ddots & & & &\ddots  & & & \\
0 & c_{i-1,i} & c_{i-1,i+1}  & \cdots & \cdots & c_{i-1,n+i-3} & c_{i-1,n+i-2} & 0 & & \\
& 0 & c_{i,i+1} & c_{i,i+2} & \cdots & \cdots & c_{i,n+i-2} & c_{i,n+i-1} & 0 & \\
& & 0 & c_{i+1,i+2} & c_{i+1,i+3} & \cdots & \cdots & c_{i+1,n+i-1} & c_{i+1,n+i} & 0 \\
 & & & & \ddots  & & & &\ddots  & 
\end{array}
$$
\caption{A frieze pattern with coefficients. \label{fig:friezecoeff}}
\end{figure}
such that $c_{i,j}\in K$, $c_{i,i+1}\neq 0$ and for the determinant of every
(complete) adjacent $2\times 2$-submatrix we have
\begin{equation}\label{eq:local}
c_{i,j} c_{i+1,j+1} - c_{i,j+1}c_{i+1,j} = c_{i+1,n+i}c_{j,j+1}.
\end{equation}
(For this definition one has to specify an algebraic structure
from which the entries can be chosen. We choose a field $K$ here but the definition applies
equally well for other algebraic structures.)

The classic frieze patterns have the property that $c_{i,i+1}=1=c_{i,n+i-1}$ for
all $i\in \mathbb{Z}$. When a frieze pattern with coefficients is tame (e.g. this is the case when
all entries $c_{i,j}$ are non-zero \cite[Proposition 2.6]{CHJ20}) then it satisfies a glide symmetry,
i.e. $c_{i,j}=c_{j,n+i}$ (see \cite[Theorem 2.12]{CHJ20}). Therefore, the region
with entries $c_{i,j}$ with $1\le i<j\le n$ forms a fundamental domain under glide reflection,
see Figure \ref{fig:friezepattern}.
\begin{figure}
$$\begin{array}{ccccc}
~0~ & ~c_{1,2}~ & ~c_{1,3}~ & ~\ldots~ & ~c_{1,n}~ \\
& 0 & c_{2,3} & \ldots & c_{2,n} \\
& & \ddots & \ddots & \vdots \\
& & & \ddots & c_{n-1,n} \\
& & & & 0 
\end{array}
$$
\caption{A fundamental domain of a tame frieze pattern with coefficients. \label{fig:friezepattern}}
\end{figure}

This gives a different and very useful viewpoint: the entries $c_{i,j}$ in this fundamental domain correspond 
bijectively to the diagonals $\{i,j\}$ of an $n$-gon, and the entries then satisfy all {\em Ptolemy relations}
(see \cite[Theorem 3.3]{CHJ20}):
$$c_{i,k}c_{j,\ell} = c_{i,\ell}c_{j,k} + c_{i,j}c_{k,\ell} \mbox{\hskip0.2cm for $1\le i\le j\le k\le \ell\le n$}
$$
see Figure \ref{fig:Ptolemy} for an illustration.
\begin{figure}
  \begin{tikzpicture}[auto]
    \node[name=s, draw, shape=regular polygon, regular polygon sides=500, minimum size=4cm] {};
    \draw[thick] (s.corner 60) to (s.corner 180);
    \draw[thick] (s.corner 180) to (s.corner 300);
    \draw[thick] (s.corner 300) to (s.corner 400);
    \draw[thick] (s.corner 400) to (s.corner 60);
    \draw[thick] (s.corner 60) to (s.corner 300);
    \draw[thick] (s.corner 180) to (s.corner 400);
    
    \draw[shift=(s.corner 60)]  node[above]  {{\small $i$}};
    \draw[shift=(s.corner 180)]  node[left]  {{\small $j$}};
    \draw[shift=(s.corner 300)]  node[below]  {{\small $k$}};
    \draw[shift=(s.corner 400)]  node[right]  {{\small $\ell$}};
   \end{tikzpicture}
   \caption{The Ptolemy relations visualized. \label{fig:Ptolemy}}
\end{figure}
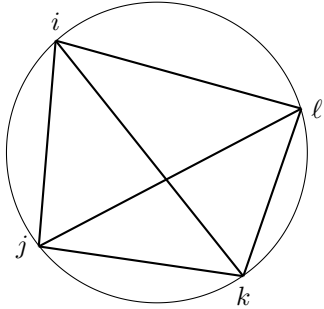
To distinguish these viewpoints we speak of a {\em frieze with coefficients} when we mean an assignment of values
to diagonals of a polygon such that all Ptolemy relations hold, and we speak of a {\em frieze pattern with
coefficients} when dealing with arrays of numbers. 

In their paper \cite{BCI74}, Broline, Crowe and Isaacs considered symmetric matrices obtained from 
(classic) frieze patterns by using the fundamental domain in Figure \ref{fig:friezepattern} as upper part and reflecting 
it along the main diagonal consisting of zeroes. They have shown that the determinant of any such 
$(n\times n)$-matrix is $-(-2)^{n-2}$. This nice formula has later been generalized to the cluster 
algebras setting by Baur and Marsh, where the entries of the matrix are cluster variables. Then the formula 
for the determinant reads $-(-2)^{n-2}c_{1,2}c_{2,3}\ldots c_{n-1,n}c_{n,1}$ \cite[Theorem 1.3]{BM12}. 
Another direction of generalization of the Broline-Crowe-Isaacs approach is to consider {\em 
generalized frieze patterns} coming from $d$-angulations of polygons (instead of triangulations for
frieze patterns) \cite{BHJ14}. Then the determinant of the corresponding symmetric matrix takes the
form $(-1)^{\ell d-1} (d-1)^{\ell}$ where $\ell$ is the number of $d$-gons in the $d$-angulation
\cite[Theorem 1.2]{BHJ14}. 

The main aim of this paper is to give a unified approach to the various generalizations of the 
Broline-Crowe-Isaacs formula appearing in the literature. 
The concept for achieving this is the notion of a {\em weak frieze}
introduced by \c{C}anak\c{c}{\i} and J{\o}rgensen \cite{CJ21}. Roughly speaking, a weak frieze 
is a map $f$
assigning values to the diagonals of a dissected polygon such that only those Ptolemy 
relations are required to hold for which one of the crossing diagonals is in the dissection. 
See Section \ref{sec:weakfrieze}
for precise definitions and more details. This yields arrays of numbers like the frieze pattern with coefficients
in Figure \ref{fig:friezecoeff}, but since not all Ptolemy relations have to be satisfied, not all
the neighbouring $2\times 2$-determinants satisfy equation (\ref{eq:local}). In this way, we get
{\em weak frieze patterns with coefficients}, e.g. the generalized frieze patterns studied in \cite{BHJ14},
see Section \ref{subsec:BHJ}. Placing the values $f(i,j)=f(j,i)$ of the diagonal $\{i,j\}$ as $(i,j)$-entry
of a symmetric matrix, we obtain a natural generalization of the Broline-Crowe-Isaacs matrices. 

Weak friezes on polygons can be glued to give new weak friezes \cite[Theorem B]{CJ21}, and
we recall this construction in Section \ref{sec:weakfrieze}. As our main result we prove in Theorem
\ref{thm:gluedet} a formula for the determinant of the symmetric matrix obtained when glueing two weak friezes.
Namely, let $P$ and $Q$ be polygons which are glued along a diagonal $d$ to a polygon $R$. 
Consider the dissection $D=\{d\}$ on $R$ and suppose that $f:\mathrm{diag}(R)\to K$ 
is a weak frieze with respect to $D$ such that $f(d)$ is invertible in $K$. 
Restricting $f$ to $P$ and $Q$ yields 
weak friezes $f_P$ and $f_Q$, and we let $M_{f}$, $M_{f_P}$, $M_{f_Q}$ be the 
corresponding symmetric matrices. Then we prove in Theorem \ref{thm:gluedet} that
$$\mathrm{det}(M_f) =  -f(d)^{-2} \cdot \mathrm{det}(M_{f_P})\cdot \mathrm{det}(M_{f_Q}).
$$
In Section \ref{sec:application} we demonstrate that our main result indeed
gives a unified approach to various results on frieze pattern determinants from the literature. 
We show how theorems by Broline-Crowe-Isaacs \cite{BCI74}, Baur-Marsh \cite{BM12} and 
Bessenrodt-Holm-J{\o}rgensen \cite{BHJ14} can all be obtained as consequences of our 
Theorem \ref{thm:gluedet}. Moreover, we discuss a recent paper by Maldonado \cite{Mal22} where the author
introduces the notion of a {\em frieze matrix}. 
The main result of \cite{Mal22} is a formula for the determinant of a frieze matrix, see 
\cite[Theorem 2.5]{Mal22} and it is claimed that this gives a generalization of the Baur-Marsh
formula \cite[Theorem 1.3]{BM12}. We shall explain in Subsection \ref{subsec:Maldonado}
that this can in fact be obtained as a 
consequence of the Baur-Marsh result. Since in turn the Baur-Marsh formula is a consequence of our
main result, this shows that also the recent result of Maldonado is a consequence of our main result
Theorem \ref{thm:gluedet}.

In addition to unifying various results on frieze determinants from the literature,
our formula in Theorem 3.3 for the determinant of the matrix coming from
a weak frieze pattern allows to compute the determinants of many further matrices. 
The new flexibility in our approach is that on each of the polygons glued to a weak 
frieze we can choose the values on the diagonals arbitrarily (whereas in 
\cite{BM12}, \cite{BCI74} all Ptolemy relations have to be satisfied,
and in \cite{BHJ14} all diagonals are assigned the value 1). 

\section{Weak friezes} \label{sec:weakfrieze}

From now on we fix a field $K$.
\smallskip

Let $P$ be a convex polygon in the Euclidean plane with vertices denoted cyclically by $\{1,2,\ldots,n\}$.
For each pair of different vertices $i\neq j$ there is a {\em diagonal} $\{i,j\}$. If $i$ and $j$ are not neighbouring 
vertices (in the cyclic ordering) then we speak of an {\em internal diagonal}. The set of all diagonals of $P$
is denoted by $\mathrm{diag}(P)$.

Two diagonals $\{i,j\}$ and $\{k,\ell\}$ are called {\em crossing} if in the cyclic ordering of vertices we have
$i<k<j<\ell$ or $i<\ell<j<k$ (note that when drawing the internal diagonals as straight lines in $P$ this means
that the diagonals cross in the interior of $P$).

A {\em dissection} $D$ of $P$ is a set of pairwise non-crossing internal diagonals of $P$ (where the case
$D=\emptyset$ is allowed). 

\begin{Definition} \label{def:weakfrieze}
Let $P$ be a polygon with vertices $\{1,2,\ldots,n\}$
and let $f:\mathrm{diag}(P)\to K$
be a map assigning values in $K$ to each diagonal in $P$. We write $f(i,j)$ for $f(\{i,j\})$. 
\begin{enumerate}
\item[{(a)}] We call the map $f$ a {\em frieze} if for all crossing diagonals $\{i,j\}$ and $\{k,\ell\}$ 
the {\em Ptolemy relation} is satisfied, that is
$$f(i,j)f(k,\ell) = f(i,k)f(j,\ell) + f(i,\ell)f(j,k).
$$
\item[{(b)}] (\c{C}anak\c{c}{\i}, J{\o}rgensen \cite{CJ21}) Let $D$ be a dissection of $P$. 
The map $f$ is called a {\em weak frieze with respect to $D$} if the Ptolemy relations are satisfied
for those pairs of crossing diagonals where one of the diagonals is in $D$.  
\end{enumerate}
\end{Definition}

A crucial property of weak friezes is that they can be glued to larger weak friezes. This is one of the main
results of \cite{CJ21}. (The result is stated in \cite{CJ21} for semifields but it goes through verbatim for
the present formulation.)

\begin{Theorem} (\c{C}anak\c{c}{\i}, J{\o}rgensen \cite[Theorem B]{CJ21}) Let $P$ be a polygon,
let $d_1,\ldots,d_m$ be pairwise non-crossing internal diagonals of $P$, dividing $P$ into subpolygons
$P_1,\ldots,P_{m+1}$, and let $D_i$ be a dissection of $P_i$ for each $i$. The disjoint union
$D=\{d_1,\ldots,d_m\}\cup D_1\cup\ldots\cup D_{m+1}$ is a dissection of $P$. 

Let $f_i:\mathrm{diag}(P_i)\to K$ be a weak frieze with respect to $D_i$ for
each $i$. Assume that if $P_i$ and $P_j$ share a diagonal $d$, then $f_i(d)=f_j(d)\neq 0$. 

Then there is a unique weak frieze $f:\mathrm{diag}(P)\to K$ with respect to $D$
such that the restriction of $f$ to $\mathrm{diag}(P_i)$ equals $f_i$ for all $i$.
\end{Theorem}

\begin{Example} \label{ex:octagon}
We consider an octagon $P$ with the dissection $D$ dividing $P$ 
into three squares as in the following figure
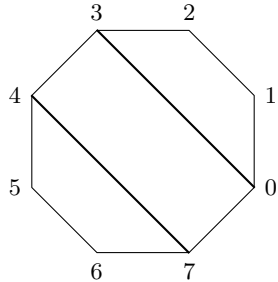
\begin{center}
\begin{tikzpicture}[auto]
    \node[name=s, draw, shape=regular polygon, regular polygon sides=8, minimum size=3.2cm] {};
    \draw[thick] (s.corner 2) to (s.corner 7);
    \draw[thick] (s.corner 3) to (s.corner 6);
   \draw[shift=(s.corner 1)] node[above] 
   {{\footnotesize $2$}};
  \draw[shift=(s.corner 2)]  node[above] 
  {{\footnotesize $3$}};
  \draw[shift=(s.corner 3)]  node[left]  
  {{\footnotesize $4$}};
  \draw[shift=(s.corner 4)]  node[left]  
  {{\footnotesize $5$}};
  \draw[shift=(s.corner 5)]  node[below] 
  {{\footnotesize $6$}};
  \draw[shift=(s.corner 6)]  node[below] 
  {{\footnotesize $7$}};
  \draw[shift=(s.corner 7)]  node[right] 
  {{\footnotesize $0$}};
  \draw[shift=(s.corner 8)]  node[right] 
  {{\footnotesize $1$}};
   \end{tikzpicture}
\end{center}
On each subpolygon $P_i$ we choose the function $f_i:\mathrm{diag}(P_i)\to K$
as being constant and equal to 1. 
Then the Ptolemy relations for the diagonals of the octagon crossing one of the diagonals $\{0,3\}$ or
$\{4,7\}$ can be used to compute $f(k,\ell) = 2^r$ where $r$ is the number of diagonals from $D$ 
crossing the diagonal $\{k,\ell\}$.
\end{Example}

\begin{Definition}
Let $P$ be a polygon with vertices $\{1,2,\ldots,n\}$ and let $D$ be a dissection of $P$. 
For every weak frieze $f:\mathrm{diag}(P)\to K$ we arrange the values
$f(i,j)$ in a triangular shape as in Figure \ref{fig:friezepattern} and perform successive glide reflections 
to give an infinite array of the form
$$
\begin{array}{cccccccccccc}
& & & \ddots & & & & & & & & \\
~0~ & ~c_{1,2}~ & ~c_{1,3}~ & ~\ldots~ & ~\ldots~ & ~c_{1,n}~ & ~0~ & & & & & \\
& 0 & c_{2,3} & \ldots & \ldots & c_{2,n} & c_{1,2} & 0 & & &  &\\
& & 0 & c_{3,4} & \ldots & c_{3,n} & c_{1,3} & c_{2,3} & 0 & & & \\ 
& & & \ddots & \ddots & \vdots & \vdots & \vdots &  & & & \\
& & & & & c_{n-1,n} & c_{1,n-1} & & & & & \\
& & & & & 0 & c_{1,n} & c_{2,n} & \ldots & c_{n-1,n} & 0 & \\
& & & & & & 0 & c_{1,2} & c_{1,3} &  \ldots & c_{1,n} & 0 \\
& & & & & & & & & \ddots & & 
\end{array}
$$
Every such array of numbers coming from a weak frieze of a dissected polygon is called
a {\em weak frieze pattern with coefficients}.
\end{Definition}

\begin{Example}
The weak frieze in Example \ref{ex:octagon}
leads to the following array of numbers 
$$\begin{array}{ccccccccccccccccc}
 & & & & \ddots &  &  &  & & & & & & & & \\
~0~ & ~1~ & ~1~ & ~1~ & ~1~ & ~2~ & ~2~ & ~1~ & ~0~ & & & & & & & \\
& 0 & 1 & 1 & 2 & 4 & 4 & 2 & 1 & ~0~ & & & & & & \\
& & 0 & 1 & 2 & 4 & 4 & 2 & 1 & 1 & ~0~ & & & & & \\
& & & 0 & 1 & 2 & 2 & 1 & 1 & 1 & 1 & ~0~ & & & & \\
& & & & 0 & 1 & 1 & 1 & 1 & 2 & 2 & 1 & ~0~ & & & \\
& & & & & 0 & 1 & 1 & 2 & 4 & 4 & 2 & 1 & ~0~ & & \\
& & & & & & 0 & 1 & 2 & 4 & 4 & 2 & 1 & 1 & ~0~ & \\
& & & & & & & 0 & 1 & 2 & 2 & 1 & 1 & 1 & 1 & ~0~ \\
 & & & &  &  &  &  & & & &  & \ddots & & & \\
\end{array}
$$
This is a generalized frieze pattern in the sense of \cite{BHJ14}. Note that all the neighbouring 
$2\times 2$-determinants in this array are 0 or 1. (A combinatorial characterization which of these
determinants are 0 and which are 1 is given in \cite[Theorem 5.1]{BHJ14}.)
\end{Example}

\section{Weak frieze determinants}
In this section we introduce symmetric matrices to weak friezes of dissected polygons. This
is a natural generalization of the construction of Broline-Crowe-Isaacs \cite{BCI74} for classic 
frieze patterns. Then we show how the determinants of these matrices behave with
respect to gluing weak friezes. 

\begin{Definition} 
Let $P$ be a polygon with vertices $\{1,2,\ldots,n\}$ (cyclically ordered) and let $D$ be a dissection of $P$. 
We associate to each weak frieze $f:\mathrm{diag}(P)\to K$
with respect to $D$ an $n\times n$-matrix 
$M_f=(m_{i,j})_{1\le i,j\le n}$
where the entries are given by
$$m_{i,j} = \left\{ \begin{array}{ll}
0 & \mbox{if $i=j$} \\
f(i,j) & \mbox{if $i\neq j$}
\end{array} \right.
$$
Note that this is a symmetric matrix since $f(i,j)=f(\{i,j\})$ is the value of $f$ on the diagonal
$\{i,j\}=\{j,i\}$. Every such matrix $M_f$ is called a {\em weak frieze matrix}.
\end{Definition}

\begin{Example} \label{ex:det}
Let $P$ be an $n$-gon and let $D=\emptyset$ be the empty dissection. Then every map
$f:\mathrm{diag}(P)\to K$ is a weak frieze with respect to $D$ 
(see Definition \ref{def:weakfrieze}\,(b)). For the constant map $f$ with $f(d)=1$ for all $d\in \mathrm{diag}(P)$
the above matrix $M_f$ has the form
$$M_f= 
\begin{pmatrix}
0 & 1 & \ldots & \ldots & 1 \\
1 & 0 & 1 & \ldots & 1 \\
\vdots & \ddots & \ddots & \ddots & \vdots \\
1 & \ldots & 1 & 0 & 1 \\
1 & \ldots & \ldots & 1 & 0
\end{pmatrix}.
$$
For later reference we state that for the determinant of this matrix we have  
$\mathrm{det}(M_f)= (-1)^{n-1} (n-1)$. In fact, replacing the last row by the sum of all rows,
taking out the factor $n-1$ and then 
subtracting the last row from each of the first $n-1$ rows gives
\begin{eqnarray*}
\mathrm{det}(M_f) & = & \mathrm{det}\begin{pmatrix}
0 & 1 & \ldots & \ldots & 1 \\
1 & 0 & 1 & \ldots & 1 \\
\vdots & \ddots & \ddots & \ddots & \vdots \\
1 & \ldots & 1 & 0 & 1 \\
n-1 & \ldots & \ldots & n-1 & n-1
\end{pmatrix} 
=  (n-1) \,\mathrm{det}\begin{pmatrix}
0 & 1 & \ldots & \ldots & 1 \\
1 & 0 & 1 & \ldots & 1 \\
\vdots & \ddots & \ddots & \ddots & \vdots \\
1 & \ldots & 1 & 0 & 1 \\
1 & \ldots & \ldots & 1 & 1
\end{pmatrix} \\
& = &  
(n-1)\,\mathrm{det}\begin{pmatrix}
-1 & 0 & \ldots & \ldots & 0 \\
0 & -1 & 0 & \ldots & 0 \\
\vdots & \ddots & \ddots & \ddots & \vdots \\
0 & \ldots & 0 & -1 & 0 \\
1 & \ldots & \ldots & 1 & 1
\end{pmatrix} 
 \,=\,  (-1)^{n-1}(n-1).
\end{eqnarray*}
\end{Example}

\begin{Theorem} \label{thm:gluedet}
For some natural numbers $r,s\ge 3$ let $P$ be an $r$-gon and $Q$ an $s$-gon which we glue to
an $(r+s-2)$-gon $R$ as in the following figure.
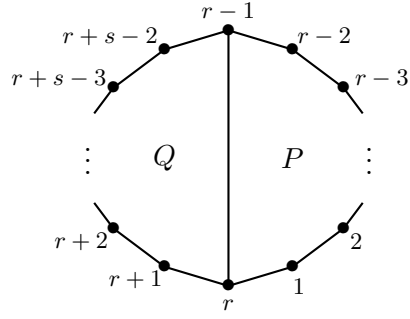
\begin{figure}[H]
\renewcommand*\figurename{Abbildung}
\centering
\begin{tikzpicture}[scale=0.85]
      \draw[thick] (0,2) -- (0,-2);
      \draw[thick] (0,-2) -- (1,-1.7);
      \draw[thick] (0,-2) -- (-1,-1.7);
       \draw[thick] (-1,-1.7) -- (-1.8,-1.1);
       \draw[thick] (1,-1.7) -- (1.8,-1.1);
      \draw[thick] (0,2) -- (1,1.7);
       \draw[thick] (1,1.7) -- (1.8,1.1);
       \draw[thick] (0,2) -- (-1,1.7);
       \draw[thick] (-1,1.7) -- (-1.8,1.1);
       \draw[thick] (-2.1,0.7) -- (-1.8,1.1);
       \draw[thick] (2.1,0.7) -- (1.8,1.1);
       \draw[thick] (2.1,-0.7) -- (1.8,-1.1);
       \draw[thick] (-2.1,-0.7) -- (-1.8,-1.1);
     
      \node (l1) at (2.2,0.1) [] {{$\vdots$}};
       \node (l4) at (-2.2,0.1) [] {{$\vdots$}};

      \node (A01) at (0,2) [] {{$\bullet$}};
      \node (A02) at (0,-2) [] {{$\bullet$}};
      \node (A03) at (1,-1.7) [] {{$\bullet$}};
      \node (A04) at (1.8,-1.1) [] {{$\bullet$}};
      \node (A05) at (-1,-1.7) [] {{$\bullet$}};
      \node (A06) at (-1.8,-1.1) [] {{$\bullet$}};
      \node (A07) at (1,1.7) [] {{$\bullet$}};
      \node (A08) at (1.8,1.1) [] {{$\bullet$}};
      \node (A09) at (-1,1.7) [] {{$\bullet$}};
      \node (A10) at (-1.8,1.1) [] {{$\bullet$}};
      
      \node (x0) at (0,-2.3) [] {{{\footnotesize $r$}}};
      \node (x1) at (1.1,-2) [] {{{\footnotesize $1$}}};
      \node (x2) at (2,-1.3) [] {{{\footnotesize $2$}}}; 
      \node (x3) at (0,2.3) [] {{{\footnotesize $r-1$}}};
      \node (x4) at (1.5,1.9) [] {{{\footnotesize $r-2$}}};
      \node (x5) at (2.4,1.25) [] {{{\footnotesize $r-3$}}}; 
      \node (x6) at (-1.45,-1.9) [] {{{\footnotesize $r+1$}}};
      \node (x7) at (-2.3,-1.3) [] {{{\footnotesize $r+2$}}}; 
       \node (x8) at (-1.85,1.9) [] {{{\footnotesize $r+s-2$}}};
      \node (x9) at (-2.65,1.25) [] {{{\footnotesize $r+s-3$}}}; 
      
      \node (x8) at (1,0) [] {{$P$}};
      \node (x9) at (-1,0) [] {{$Q$}}; 
\end{tikzpicture}
\end{figure}
On the polygon $R$ we consider the dissection $D=\{(r-1,r)\}$ only consisting of the diagonal along which 
$P$ and $Q$ are glued. Let $f:\mathrm{diag}(R)\to K$ be a weak frieze on $R$ with
respect to $D$ and suppose that $f(r-1,r)$ is non-zero. 
By restriction of $f$ to the diagonals of $P$ and $Q$ we get weak friezes $f_P$ and
$f_Q$, respectively. Let $M_f$, $M_{f_P}$ and $M_{f_Q}$ be the matrices associated to these weak friezes.
Then for the determinants the following formula holds:
$$\mathrm{det}(M_f) = - f(r-1,r)^{-2}\cdot \mathrm{det}(M_{f_P})\cdot \mathrm{det}(M_{f_Q}).
$$  
\end{Theorem}

\smallskip

\begin{proof}
For abbreviation we set $c:=f(r-1,r)$, then the matrix $M_f$ has the following shape 
$$M_f = \left( \begin{array}{ccc|cc|ccc}
 & &  &  \ast & \ast & & &   \\
  & M' &  & \vdots & \vdots &  & \ast &  \\
   & &  & \ast & \ast & & & \\ \hline
  \ast   & \ldots & \ast  & 0 & c & \ast & \ldots & \ast \\
     \ast & \ldots & \ast & c &0 &  \ast & \ldots & \ast \\ \hline
      & &  & \ast  & \ast & &  &  \\
       & \ast &  &   \vdots & \vdots & & N' &  \\
       & &  &  \ast & \ast & & &   \\
  \end{array} \right)
  $$
where the matrix $M'$ together with the highlighted rows and columns $r-1$ and $r$ is the matrix
$M_{f_P}$ and similarly, the matrix $N'$ together with rows and columns $r-1$ and $r$ is the matrix
$M_{f_Q}$. 

The Ptolemy relations for the weak frieze $f$ are as follows: for every
$1\le i\le r-2$ and every $r+1\le j\le r+s-2$ we have
$$f(r-1,r) f(i,j) = f(i,r-1)f(r,j)+f(i,r)f(r-1,j).
$$
Denoting the matrix entries by $M_f=(x_{k,\ell})$,
these Ptolemy relations read
\begin{equation} \label{eq:Ptolemy}
x_{i,j} = (c^{-1} x_{i,r-1}) x_{r,j} + (c^{-1}x_{i,r}) x_{r-1,j}.
\end{equation}
On the matrix $M_f$ we perform the following row operations, leading to the matrix
$\widetilde{M}_f=(\widetilde{x}_{k,\ell})$ (where $\ast$ stands for any column index):
$$\widetilde{x}_{i,\ast} = \left\{
\begin{array}{ll}
x_{i,\ast} - (c^{-1}x_{i,r-1}) x_{r,\ast} - (c^{-1}x_{i,r}) x_{r-1,\ast} & \mbox{~~for $1\le i \le r-2$} \\
x_{i,\ast} & \mbox{~~for $r-1\le i \le r+s-2$}
\end{array}
\right.
$$
That is, we perform row operations on the rows belonging to the subpolygon $P$, but leave the rows
coming from $Q$ unchanged. For the entries of the new matrix $\widetilde{M}_f$ the following holds
for all $1\le i\le r-2$:
\begin{enumerate}
\item[{(i)}] $\widetilde{x}_{i,r-1} = x_{i,r-1} - (c^{-1}x_{i,r-1}) \underbrace{x_{r,r-1}}_c-
(c^{-1}x_{i,r})\underbrace{x_{r-1,r-1}}_0
= 0$
\item[{(ii)}] $\widetilde{x}_{i,r} = x_{i,r} - (c^{-1}x_{i,r-1}) \underbrace{x_{r,r}}_0-
(c^{-1}x_{i,r})\underbrace{x_{r-1,r}}_c
= 0$ 
\item[{(iii)}] $\widetilde{x}_{i,j}=0$ for every $r+1\le j\le r+s-2$ by the Ptolemy relations in Equation 
(\ref{eq:Ptolemy}). 
\end{enumerate}
This means that the new matrix $\widetilde{M}_f$ is of the form
$$\widetilde{M}_f = \left( \begin{array}{ccc|cc|ccc}
 & &  &  0 & 0 & & &   \\
  & \widetilde{M}' &  & \vdots & \vdots &  & 0 &  \\
   & &  & 0 & 0 & & & \\ \hline
  \ast   & \ldots & \ast  & 0 & c & \ast & \ldots & \ast \\
     \ast & \ldots & \ast & c &0 &  \ast & \ldots & \ast \\ \hline
      & &  & \ast  & \ast & &  &  \\
       & \ast &  &   \vdots & \vdots & & N' &  \\
       & &  &  \ast & \ast & & &   \\
  \end{array} \right)
  $$
  for a suitable matrix $\widetilde{M}'$. Note that this matrix has a (lower) block triangular shape
  with matrices $\widetilde{M}'$ and $M_{f_Q}$ as diagonal blocks. Moreover, also 
  the matrix $\widetilde{M}_{f_P}$ consisting of the first $r$ rows and columns of $\widetilde{M}_{f}$
  has a block triangular shape, with $\widetilde{M}'$ and 
  $\begin{pmatrix} 0 & c \\ c & 0 \end {pmatrix}$ as blocks on the diagonal. Since 
  $\widetilde{M}_{f_P}$ has been obtained from $M_{f_P}$ by row operations, the two
  matrices have the same determinant.
  We can then conclude
  \begin{eqnarray*}
  \mathrm{det}(M_f) & = & \mathrm{det}(\widetilde{M}_f)
   =  \mathrm{det}(\widetilde{M}')\cdot \mathrm{det}(M_{f_Q}) 
   =  - c^{-2} \mathrm{det}(M_{f_P}) \cdot \mathrm{det}(M_{f_Q})
  \end{eqnarray*}
  and this completes the proof of the theorem. 
\end{proof}

\section{Applications} \label{sec:application}
In this section we show how several results on frieze pattern determinants from the 
literature can be obtained as consequences of our main result Theorem \ref{thm:gluedet}.

\subsection{Frieze determinants of generalized frieze patterns (Bessenrodt-Holm-J{\o}rgensen \cite{BHJ14})}
\label{subsec:BHJ}
In the paper \cite{BHJ14} the authors introduce and study {\em generalized frieze patterns} which come
from $d$-angulations of polygons. In Definition 3.1 of \cite{BHJ14} a combinatorial procedure is described
how these generalized frieze patterns can be viewed as weak friezes. Namely, they are obtained as weak 
friezes glued from pieces (that is, weak friezes on subpolygons) where on each of the pieces the
weak frieze is constant and equal to 1. In this way, the construction is not restricted to $d$-angulations but 
pieces of different sizes can be glued to yield weak friezes on arbitrary dissections. 

Let $P_d$ be a $d$-gon (where $d\ge 3$) and $D=\emptyset$ the empty dissection,
and let $f:\mathrm{diag}(P_d)\to K$ be the weak frieze
which is constant and equal to 1. 
In Example \ref{ex:det} we have seen that for the determinant of the corresponding frieze matrix we have  
$\mathrm{det}(M_f)= (-1)^{d-1} (d-1)$.

Now let $P$ be an $n$-gon and let $D$ be a dissection of $P$ which divides $P$ into subpolygons 
$P_1,\ldots,P_{\ell}$ of sizes $d_1,\ldots,d_{\ell}$. Let $f:\mathrm{diag}(P)\to K$
be the weak frieze obtained by gluing the weak friezes $f_i:\mathrm{diag}(P_i)\to K$
which are constant and equal to 1. By applying Theorem \ref{thm:gluedet} inductively, 
the weak frieze matrix $M_f$ is the
product of the determinants of the weak frieze matrices $M_{f_i}$, up to a sign $(-1)^{\ell -1}$; in fact,
each gluing produces a sign and with $\ell$ pieces we have $\ell -1$ gluings; also note that the factor 
$f(r-1,r)^{-2}$ in Theorem \ref{thm:gluedet} is equal to 1 by assumption on the weak friezes $f_i$. Therefore,
from the above considerations we obtain the following consequence of our main result. For the second equation 
below use that $n=(\sum_{i=1}^{\ell} d_i)-2(\ell -1)$.

\begin{Corollary}
With the above notation we have
$$\mathrm{det}(M_f) = (-1)^{\ell -1}\cdot \prod_{i=1}^{\ell} (-1)^{d_i-1} (d_i-1)
= (-1)^{n-1}\prod_{i=1}^{\ell} (d_i-1).
$$
\end{Corollary}

In the special case where $d_1=\ldots=d_{\ell}=d$ for some $d\ge 3$, that is for the case of a
$d$-angulation, the above corollary recovers \cite[Theorem 1.2]{BHJ14}.

\subsection{A frieze pattern determinant related to cluster algebras (Baur-Marsh \cite{BM12})}
\label{sec:BM}
For a cluster algebra of Dynkin type $A$ there is a combinatorial model given by triangulations of
polygons. The diagonals of the triangulation form the initial seed, the cluster variables 
correspond to the diagonals and they are obtained from the initial cluster variables by Ptolemy relations.    
Following this viewpoint, Baur and Marsh \cite{BM12} consider a triangulation $\pi$ 
of a regular $n$-gon $P_n$ 
and they associate to each diagonal $\{i,j\}$ of $P_n$ 
the corresponding cluster variable $u_{i,j}$. In our language of weak friezes this is a 
map $f:\mathrm{diag}(P_n)\to K$ where $f(i,j)=u_{i,j}$ assigns to each
diagonal the corresponding cluster variable. 
Then for the 
$n\times n$-matrix $U(\pi)=(u_{i,j})=(f(i,j))$ Baur and Marsh obtain a nice formula for the determinant.

\begin{Corollary} \label{cor:BM} (\cite[Theorem 1.3]{BM12}) With the above notation we have 
$$\mathrm{det}(U(\pi)) = -(-2)^{n-2} f(1,2)f(2,3)\cdot \ldots \cdot f(n-1,n)f(n,1).
$$
\end{Corollary}

We shall show how this result can be obtained as a consequence from our main result Theorem \ref{thm:gluedet}.
Any triangulation $\pi$ of a regular $n$-gon $P_n$ can be obtained by successively gluing triangles.
Note that each triangle is a frieze (not only a weak frieze) because there are no internal diagonals. 
We proceed by induction on $n\ge 3$. For $n=3$, a (weak) frieze on a triangle gives a (weak) frieze matrix 
$$\begin{pmatrix} 0 & f(1,2) & f(1,3) \\ f(1,2) & 0 & f(2,3) \\ f(1,3) & f(2,3) & 0 
\end{pmatrix}. 
$$ 
Its determinant can easily be computed as $2f(1,2)f(2,3)f(1,3)$. Since $f(1,3)=f(3,1)$ (the 
cluster variables belong to unordered diagonals) this yields the formula in the corollary. 

Inductively, suppose any (weak) frieze $f:\mathrm{diag}(P_n)\to K$ satisfies that 
its frieze matrix $M_f$ has
$$\mathrm{det}(M_f) = -(-2)^{n-2} f(1,2)f(2,3)\cdot \ldots \cdot f(n-1,n)f(n,1).
$$

In the inductive step, we glue another triangle to the polygon $P_n$ and the (weak) frieze $f$
along the diagonal of $P_n$ with vertices $1$ and $n$. The new polygon $P_{n+1}$ with new
vertex $n+1$ yields a new (weak) frieze $f':\mathrm{diag}(P_{n+1})\to K$ 
whose restriction to $P_n$ is equal to $f$. By Theorem \ref{thm:gluedet} and the induction hypothesis, 
the new (weak) frieze matrix satisfies
\begin{eqnarray*}
\mathrm{det}(M_{f'}) & = &  -f(1,n)^{-2}\cdot (2f(1,n)f'(n,n+1)f'(n+1,1))\cdot \mathrm{det}(M_f)\\
& = & -f(1,n)^{-2}(2f(1,n)f'(n,n+1)f'(n+1,1))\cdot \\
& & \hskip1cm \cdot (-(-2)^{n-2}f(1,2)f(2,3)\cdot \ldots\cdot f(n-1,n)f(n,1)) \\
& = & 2\cdot (-2)^{n-2}f(1,2)f(2,3)\cdot \ldots \cdot f(n-1,n)f'(n,n+1)f'(n+1,1) \\
& = & -(-2)^{n-1} f'(1,2)f'(2,3)\cdot \ldots\cdot f'(n-1,n)f'(n,n+1)f'(n+1,1)
\end{eqnarray*}
and this proves the induction step. 
\smallskip

The first authors who proved a result on frieze pattern determinants were Broline, Crowe and Isaacs
\cite{BCI74}. Their result is about frieze pattern determinants for classic frieze patterns. That is, 
all the values $f(i,i+1)$ are equal to 1. The result in \cite{BCI74} states that the determinant of
such a frieze matrix is equal to $-(-2)^{n-2}$, and this is an immediate special case of
Corollary \ref{cor:BM}.

\subsection{Determinants of frieze matrices (Maldonado \cite{Mal22})}
\label{subsec:Maldonado}
In a recent paper \cite{Mal22}, Maldonado introduces the notion of {\em frieze matrix} and 
he proves a formula for their determinants, see
\cite[Theorem 2.5]{Mal22}. It is claimed in \cite{Mal22} that this result generalizes the
formula by Baur and Marsh in \cite[Theorem 1.3]{BM12}. 
The aim of this section is to explain that the result of Maldonado 
can be obtained from the Baur-Marsh formula.
Since we have seen in Section \ref{sec:BM} that the Baur-Marsh formula can be obtained as a consequence of our
Theorem \ref{thm:gluedet} this means that also the formula of Maldonado is a consequence our
Theorem \ref{thm:gluedet}.

We start by recalling the relevant background from \cite{Mal22}. 

\begin{Definition} \label{def:friezematrix} (\cite[Definition 2.1]{Mal22})
A symmetric $n\times n$-matrix $C=(c_{i,j})$ over $K$ is called a {\em frieze matrix}
if $c_{i,j}=0$ if and only if $i=j$ and the entries satisfy the {\em generalized diamond rule}
$$c_{i,j} c_{i+1,j+1} - c_{i+1,j}c_{i,j+1} = c_{i,i+1} c_{j,j+1}$$
for all $j\ge i+1$.
\end{Definition}

In fact, we interpret the definition in \cite{Mal22} to mean that the generalized diamond rule should apply for all
$n-1\ge j\ge i+1\ge 2$. 

Then the main result of \cite{Mal22} is the following formula for the determinant of a frieze matrix.

\begin{Theorem} \label{thm:maldonado} (\cite[Theorem 2.5]{Mal22}) If $C=(c_{i,j})$ is a frieze matrix then
$$\mathrm{det}(C) = -(-2)^{n-2} c_{1,n} \prod_{i=1}^{n-1} c_{i,i+1}.
$$
\end{Theorem}

We can use the upper part of a frieze matrix $C=(c_{i,j})$ as above to form part of the shape of 
a frieze pattern with coefficients:
$$\begin{array}{ccccc}
0 & c_{1,2} & c_{1,3} & \ldots & c_{1,n} \\
& 0 & c_{2,3} & \ldots & c_{2,n} \\
& & \ddots & & \vdots \\
& & & 0 & c_{n-1,n} \\
& & & & 0
\end{array}
$$
The generalized diamond rule in Definition \ref{def:friezematrix} states that for the determinant of
each neighbouring 
$2\times 2$-matrix we have
$$\mathrm{det}\begin{pmatrix}
c_{i,j} & c_{i,j+1} \\ c_{i+1,j} & c_{i+1,j+1} 
\end{pmatrix} = c_{i,i+1}c_{j,j+1},
$$
the product of the entries in the first non-trivial diagonal in rows $i$ and $j$. Moreover, the
conditions on the indices in Definition \ref{def:friezematrix} mean that this condition holds for all
$2\times 2$-matrices entirely contained in the above triangular shape of the upper part of $C$. 
Note that there is no condition involving the last column of the matrix $C$. However, we shall show now 
that suitable conditions follow from the ones given in Definition \ref{def:friezematrix}. 

\begin{Proposition} \label{prop:overlap}
Let $C=(c_{i,j})$ be a frieze matrix as in Definition \ref{def:friezematrix}. Then for all $1\le i\le n-1$
we have
$$c_{i,n}c_{1,i+1}-c_{i+1,n}c_{1,i} = c_{i,i+1}c_{1,n}.
$$
\end{Proposition}

\begin{proof}
We proceed by induction on $i$. For $i=1$ we have
$$c_{1,n}c_{1,2}-c_{2,n}c_{1,1} = c_{1,2}c_{1,n}
$$
since $c_{1,1}=0$ by Definition \ref{def:friezematrix}.
As induction hypothesis (IH) suppose that for some $2\le i\le n-3$ we have
$$c_{i-1,n}c_{1,i}-c_{i,n}c_{1,i-1} = c_{i-1,i}c_{1,n}
$$

For carrying out the induction step we first note that 
\begin{equation} \label{eq:step1}
c_{1,i+1}c_{i-1,i} = -c_{1,i-1}c_{i,i+1} + c_{1,i} c_{i-1,i+1}
\end{equation}
and 
\begin{equation} \label{eq:step2}
c_{i+1,n}c_{i-1,i} = -c_{i,i+1}c_{i-1,n} + c_{i-1,i+1} c_{i,n}.
\end{equation}
These are special cases of \cite[Lemma 2.2]{Mal22} (see also \cite[Theorem 3.3]{CHJ20}), 
namely for the case $1\le i-1\le i\le i+1$ and for the case $i-1\le i\le i+1\le n$, respectively.
We then compute
\begin{eqnarray*}
c_{i,n}c_{1,i+1}-c_{i+1,n}c_{1,i} & \stackrel{(\ref{eq:step1})}{=} & 
c_{i,n}( -c_{1,i-1}\frac{c_{i,i+1}}{c_{i-1,i}} + c_{1,i} \frac{c_{i-1,i+1}}{c_{i-1,i}}) - c_{i+1,n}c_{1,i} \\
& \stackrel{(\ref{eq:step2})}{=} & c_{i,n}( -c_{1,i-1}\frac{c_{i,i+1}}{c_{i-1,i}} + c_{1,i} \frac{c_{i-1,i+1}}{c_{i-1,i}}) \\
& & \hskip0.7cm -c_{1,i}(-\frac{c_{i,i+1}}{c_{i-1,i}}c_{i-1,n} + \frac{c_{i-1,i+1}}{c_{i-1,i}}c_{i,n}) \\
& = & -c_{i,n}c_{1,i-1} \frac{c_{i,i+1}}{c_{i-1,i}} + c_{1,i}c_{i-1,n} \frac{c_{i,i+1}}{c_{i-1,i}} \\
& = & (-c_{i,n}c_{1,i-1}  + c_{1,i}c_{i-1,n})\frac{c_{i,i+1}}{c_{i-1,i}} \\
& \stackrel{(IH)}{=} & c_{i-1,i}c_{1,n}\frac{c_{i,i+1}}{c_{i-1,i}} \\
& = & c_{i,i+1}c_{1,n}
\end{eqnarray*}
and this yields the induction step. 
\end{proof}

We can extend the triangular shape of the upper part of 
any frieze matrix by glide reflection to get a doubly infinite pattern of the form
$$
\begin{array}{cccccccccccc}
& & & \ddots & & & & & & & & \\
~0~ & ~c_{1,2}~ & ~c_{1,3}~ & ~\ldots~ & ~\ldots~ & ~{\blue c_{1,n}}~ & ~0~ & & & & & \\
& 0 & c_{2,3} & \ldots & \ldots & {\blue c_{2,n}} & {\red c_{1,2}} & 0 & & &  &\\
& & 0 & c_{3,4} & \ldots & {\blue c_{3,n}} & {\red c_{1,3}} & c_{2,3} & 0 & & & \\ 
& & & \ddots & \ddots & \vdots & \vdots & \vdots &  & & & \\
& & & & & {\blue c_{n-1,n}} & {\red c_{1,n-1}} & & & & & \\
& & & & & 0 & {\red c_{1,n}} & c_{2,n} & \ldots & c_{n-1,n} & 0 & \\
& & & & & & 0 & c_{1,2} & c_{1,3} &  \ldots & c_{1,n} & 0 \\
& & & & & & & & & \ddots & & 
\end{array}
$$
This pattern is a frieze with coefficients 
(in the sense of \cite{CHJ20}). The crucial consequence of Proposition \ref{prop:overlap} is that
also the neighbouring $2\times 2$-matrices at the overlap (blue and red columns) satisfy the 
required condition for their determinant. But then a frieze matrix in the sense of \cite{Mal22}
is merely a fundamental domain
of a frieze pattern with coefficients. The latter has a combinatorial model by a triangulation of a polygon, 
assigning values to the diagonals such that all Ptolemy relations are satisfied. For the
corresponding frieze pattern determinants, the set-up in the paper by Baur and Marsh \cite{BM12}, with 
indeterminates as values on the diagonals, is the most general one. Therefore, the main result  
in \cite{Mal22}, restated here as Theorem \ref{thm:maldonado}, 
follows from the Baur-Marsh formula (stated here as Corollary 
\ref{cor:BM}). And as we have shown in Section \ref{sec:BM}, the Baur-Marsh formula is 
a special case of our main result Theorem \ref{thm:gluedet}. So we obtain the Baur-Marsh formula and
the Maldonado formula as special cases of our main result in this paper.



\begin{thebibliography}{19}

\bibitem{BM12}
K.\ Baur, R.\ J.\ Marsh,
{\em Categorification of a frieze pattern determinant},
J. Combin. Theory Ser. A 119 (2012), 1110-1122.

\bibitem{BHJ14}
C.\ Bessenrodt, T.\ Holm, P.\ J{\o}rgensen,
{\em Generalized frieze pattern determinants and higher angulations of polygons},
J. Combin. Theory Ser. A 123 (2014), 30-42.

\bibitem{BCI74}
D.\ Broline, D.\ W.\ Crowe, I.\ M.\ Isaacs,
{\em The geometry of frieze patterns}, 
Geo\-metriae Dedicata 3 (1974), 171-176.

\bibitem{CJ21}
{\.I}.\ \c{C}anak\c{c}{\i}, P.\ J{\o}rgensen,
{\em Friezes, weak friezes, and T-paths},
Adv. in Appl. Math. 131 (2021), Paper No. 102253. 

\bibitem{CC73}
J.\ H.\ Conway, H.\ S.\ M.\ Coxeter, 
{\em Triangulated polygons and frieze patterns}, Math. Gaz. 57 (1973), no. 400, 87-94 and   
no. 401, 175-183.

\bibitem{Cox71}
H.\ S.\ M.\ Coxeter, 
{\em Frieze patterns}, Acta Arith. 18 (1971), 297-310.

\bibitem{CHJ20}
M.\ Cuntz, T.\ Holm, P.\ J{\o}rgensen,
{\em Frieze patterns with coefficients},
Forum Math. Sigma 8 (2020), e17.

\bibitem{FZ02}
S.\ Fomin, A.\ Zelevinsky,
{\em Cluster algebras. I. Foundations},
J. Amer. Math. Soc. 15 (2002), 497-529.

\bibitem{HJ17}
T.\ Holm, P.\ J{\o}rgensen, 
{\em A $p$-angulated generalisation of Conway and Coxeter's 
theorem on frieze patterns},
Int. Math. Res. Not. (IMRN) 2020 (2020), 71-90.

\bibitem{Mal22}
J.\ P.\ Maldonado,
{\em Frieze matrices and infinite frieze patterns with coefficients},
Preprint (2022), arXiv:2207.04120.

\bibitem{MG15}
S.\ Morier-Genoud, {\em Coxeter's frieze patterns at the crossroads of algebra, 
geometry and combinatorics}, 
Bull. Lond. Math. Soc. 47 (2015), no. 6, 895-938.

\end{thebibliography}
\end{document}